\documentclass[letterpaper, 10 pt, conference]{ieeeconf}  % Comment this line out
\IEEEoverridecommandlockouts                              % This command is only
\usepackage{TLC}

\usepackage{amsmath}
\usepackage[style=ieee,
            url=false,
            eprint=false,
            date=year]{biblatex}
\title{\LARGE \bf
Relaxation systems and cyclic monotonicity
}

\author{Thomas Chaffey*, Henk J. van Waarde*, Rodolphe Sepulchre
\thanks{*These authors contributed equally.}
\thanks{T. Chaffey is with the University of Cambridge, Department of Engineering,
Trumpington Street, CB2 1PZ, \texttt{tlc37@cam.ac.uk}. H. J. van Waarde is with the Bernoulli Institute for Mathematics, Computer
Science, and Artificial Intelligence, University of Groningen, Nijenborgh 9,
9747 AG, Groningen, Netherlands, \texttt{h.j.van.waarde@rug.nl}.
R. Sepulchre is with KU Leuven,
Department of Electrical Engineering (STADIUS),
KasteelPark Arenberg, 10,
B-3001 Leuven, Belgium,
\texttt{rodolphe.sepulchre@kuleuven.be}.}
\thanks{The research leading to these results has received funding from the European Research Council under the
Advanced ERC Grant Agreement SpikyControl n.101054323.  The work of T. Chaffey was supported by Pembroke College, Cambridge.
}}
\begin{document}

\maketitle
\thispagestyle{empty}
\pagestyle{empty}

%%%%%%%%%%%%%%%%%%%%%%%%%%%%%%%%%%%%%%%%%%%%%%%%%%%%%%%%%%%%%%%%%%%%%%%%%%%%%%%%
\begin{abstract}
        It is shown that an LTI system is a relaxation system if and only if its
        Hankel operator is cyclic monotone.  Cyclic monotonicity of the Hankel
        operator implies the existence of a storage function whose gradient is the
        Hankel operator.  This storage is a function of past inputs alone, is
        independent of the state space realization, and admits a generalization to
        nonlinear circuit elements. 
\end{abstract}

%%%%%%%%%%%%%%%%%%%%%%%%%%%%%%%%%%%%%%%%%%%%%%%%%%%%%%%%%%%%%%%%%%%%%%%%%%%%%%%%
\section{Introduction}

Relaxation systems are a class of LTI systems which first arose in the study of relaxation phenomena in viscoelastic
materials, and, in the finite dimensional case, correspond to RC and RL circuits
\autocite{Meixner1964}.  Relaxation systems are highly structured.  They correspond
to systems with completely monotonic impulse responses, with transfer functions which are sums of first order lags \autocite{Bernstein1928,
Widder1941, Meixner1964} and it was shown by \textcite{Willems1972a} that they admit
state space realizations which are both externally symmetric, corresponding to the
circuit property of reciprocity, and internally symmetric, encoding the fact that all
the energy storage elements are of the same type.  
There has been a recent revival of interest in relaxation systems
\autocite{Pates2019, Pates2022, Grussler2022, Grussler2022a, Bar-Shalom2023, Yafaev2015,
Yafaev2015a, Margaliot2019}. For example, it was observed by Pates \emph{et al.} 
\autocite{Pates2019, Pates2022} that they
admit very simple $H_\infty$-optimal controllers, with highly structured circuit
realizations.

Dissipativity theory \autocite{Willems1972} connects the circuit theory of passivity to the dynamical
systems theory of stability via the \emph{storage function}, which represents the
energy stored in a system.  
For a relaxation system, there exists a 
storage function which is completely determined by the Hankel operator, that is, the
future output in response to a past input \autocite{Willems1972a}. Relaxation 
systems therefore represent a class of systems for which
the storage can be defined externally, as a function of past input only.

Existing characterizations of relaxation systems rely on linearity
and time invariance.
We are motivated by a characterization that is not 
limited to LTI systems. This paper presents some preliminary steps in this direction,
through connections to monotone operator theory.  The property of
monotonicity was originally introduced in efforts to generalize the property of
passivity to networks of nonlinear resistors \autocite{Duffin1946, Minty1960,
Minty1961, Minty1962}.  Monotone operator theory now forms a pillar of convex
optimization theory \autocite{Rockafellar1976, Bauschke2011, Combettes2018, Ryu2022a},
owing to the fact that the gradient of a convex function is a monotone operator.

An early question in the theory of monotone operators was when the converse is true,
when is a monotone operator the gradient of a convex function?  This question was
answered by \textcite{Rockafellar1966, Rockafellar1970b}, who showed that a stronger
property than monotonicity is required: cyclic monotonicity.

In this paper, we reconnect the property of cyclic monotonicity with its
circuit theoretic origins, showing that cyclic monotonicity corresponds precisely to
relaxation, that is, to circuits with a single type of energy storage element.  Our
main result shows that an equivalent characterization of relaxation is that a
system's Hankel operator is cyclic monotone. For single input, single output LTI
operators, this equivalence was shown independently in the recent work of Yafaev
\autocite{Yafaev2015, Yafaev2015a}. Our proof is MIMO, and uses a state space
representation.  Cyclic monotonicity of the Hankel operator
implies that it is the gradient of some convex functional, and we show that this convex
functional is precisely the intrinsic storage of a relaxation system observed by
Willems.  Because cyclic monotonicity is not restricted to linear systems, our 
characterization opens the way to a nonlinear concept of relaxation.

Cyclic monotonicity has previously been studied in the context of Lur'e
systems \autocite{Adly2017a, Scherer2022}, multi-agent systems
\autocite{Sharf2019} and recently in the context of incrementally
port-Hamiltonian systems \autocite{Camlibel2023}, where it was shown that
a port-Hamiltonian system with a maximal cyclic monotone Dirac structure may be defined
in terms of a convex function of the state and input.  In contrast, we consider
cyclic monotonicity of an \emph{external} map, the Hankel operator, that maps past inputs to future outputs.

The remainder of this paper is structured as follows.  In Section~\ref{sec:prelims},
we introduce the necessary preliminary material from the theory of passivity and
monotone operators.  
In Section~\ref{sec:relaxation_cyclic}, we give the first of our
main results, that relaxation is equivalent to cyclic monotonicity of the Hankel
operator.  In Section~\ref{sec:LTI_storage}, we introduce a new notion of an
intrinsic storage functional and show that the convex functional whose
gradient is the Hankel operator is the intrinsic storage of Willems. Conclusions 
and directions for future work are given in Section~\ref{sec:conclusions}.

\section{Preliminaries}\label{sec:prelims}

\subsection{State space systems and Hankel operators}

We study linear, time-invariant state space systems of the form
\begin{IEEEeqnarray*}{rCl}
        \dot x(t) &=& A x(t) + B u(t)\IEEEyesnumber \label{eq:system}\\
             y(t) &=& C x(t) + D u(t),
\end{IEEEeqnarray*}
where $x(t) \in \R^n$, $u(t) \in \R^m$, $y(t) \in \R^p$ $A \in \R^{n\times n}$,  $B \in \R^{n
\times m}$, $C \in \R^{p \times n}$ and $D \in \R^{p \times m}$.
A system is said to be \emph{stable} if $A$ is Hurwitz, and \emph{minimal} if 
$(A, B)$ is controllable and $(A, C)$ is observable.  The transfer function of 
system \eqref{eq:system} is given by $H(s) := C(sI - A)^{-1}B + D$, and the impulse response is given by
$
h(t) := D\delta(t) + Ce^{A t}B, 
$ where $\delta(t)$ denotes the Dirac delta.  We also define $g(t) := Ce^{A t}B$ to be
the impulse response of the system with no feedthrough term.

A complete inner product space is called a Hilbert space.
The space $L_{2}(\R, \R^n)$ is the set of signals $u:
\R \to \R^n$ such that
\begin{IEEEeqnarray*}{rCl}
\int_{-\infty}^\infty u(t)\tran u(t) \dd{t} < \infty.
\end{IEEEeqnarray*}
This space forms a Hilbert space of equivalence classes of
functions when equipped with the inner product
\begin{IEEEeqnarray*}{rCl}
\ip{u}{y} := \int_{-\infty}^\infty u(t)\tran y(t) \dd{t},
\end{IEEEeqnarray*}
which induces the norm $\norm{u} := \sqrt{\ip{u}{u}}$.  We define $L_2(\R_{\geq 0},
\R^n)$ and
$L_2(\R_{\leq 0}, \R^n)$ similarly, but with  time axes of $[0, \infty)$ and  $(-\infty, 0]$,
respectively.  We will use the shorthand notation $L_2^n$ for $L_2(\R_{\geq 0},
\R^n)$.

A stable system admits a \emph{Hankel operator}, which maps an input
on $L_{2}(\R_{\leq 0}, \R)$ to the corresponding output on $L_{2}(\R_{\geq 0}, \R)$,
assuming zero input from time $0$.  Given an impulse response $h$ and input $\bar{u} \in L_{2}(\R_{\leq 0}, \R)$, the
output of the Hankel operator $\Gamma_h$ at time $t$ is given by
\begin{IEEEeqnarray*}{rCl}
y(t) &=& \int_{-\infty}^0 h(t - \tau)\bar{u}(\tau) \dd{\tau}.
\end{IEEEeqnarray*}
Letting $u(t) := \bar{u}(-t)$, 
the Hankel operator has the expression
\begin{IEEEeqnarray*}{rCl}
\left(\Gamma_h u\right)(t) &:=& \int_0^\infty h(t + \tau)u(\tau) \dd{\tau},
\end{IEEEeqnarray*}
and defines an operator on $L_{2}(\R_{\geq 0}, \R)$.  If the system is stable, the
Hankel operator is continuous \autocite[Prop. 4.1]{Partington1988}.

For the remainder of this paper, we will consider systems which are \emph{square},
that is, the input dimension $m$ is equal to the output dimension $p$.

%study of circuits \autocite{Hughes2019}, circuits of the relaxation type may always
%be completely characterized by a minimal state space realization
%\autocite{Pates2022}.

\subsection{Passivity, reciprocity and relaxation}

Passivity is a formalization of the notion that a system can be realized without any
internal power source.  Central to the theory of passivity is the storage function,
which represents the energy stored within a system. We recall the following
definition of passivity.

\begin{definition}[\protect{\autocite[Def. 5]{Hughes2017c}}]\label{def:passive}
        A system of the form \eqref{eq:system} is said to be \emph{passive} if, for
        any input/output trajectory $(u, y)$ of the system and $t_0 \in \R$, there
        exists a $K \in \R$ such that, if $(\hat u, \hat y)$ is also an input/output
        trajectory of the system and $(\hat u(t), \hat y(t)) = (u(t), v(t))$ for all
        $t < t_0$, then
        \begin{IEEEeqnarray*}{rCl}
                -\int_{t_0}^{t_1} \hat u(t)\tran \hat y(t) \dd{t} \leq K
        \end{IEEEeqnarray*}
        for all $t_1 \geq t_0$.
\end{definition}

It is shown in \autocite[Thm. 13]{Hughes2017c} that, for a (not necessarily minimal)
system of the form \eqref{eq:system}, Definition~\ref{def:passive} is
equivalent to the existence of a matrix $Q = Q\tran \succeq 0$ satisfying the linear matrix
inequality
\begin{IEEEeqnarray}{rCl}
        \begin{pmatrix} A\tran Q + Q A & Q B - C \tran \\
                        B \tran Q - C & - D - D\tran \end{pmatrix} \preceq
                        0.\label{eq:lmi}
\end{IEEEeqnarray}
This is precisely the condition given by \autocite[Thm. 3]{Willems1972a} in the
context of minimal LTI state space systems.

A signature matrix is a diagonal matrix whose diagonal entries are either $1$ or
$-1$.
\begin{definition}
        A system of the form \eqref{eq:system} is said to be \emph{(externally)
        reciprocal} with respect to the signature matrix $\Sigma_e$ if  $\Sigma_e
        H(s) = \Sigma_e H(s)\tran$, where $H(s)$ is the transfer matrix of
        \eqref{eq:system}.
\end{definition}

Reciprocal systems admit internally reciprocal state space realizations.
\begin{theorem}[\protect{\autocite[Thm. 6]{Willems1972a}}]
        A system of the form \eqref{eq:system} is reciprocal if and only if it admits a state space realization
        $(A, B, C, D)$ such that
        \begin{IEEEeqnarray*}{rCl}
        \begin{pmatrix} \Sigma_i & 0 \\ 0 & \Sigma_e \end{pmatrix}
        \begin{pmatrix} -A & -B \\ C & D \end{pmatrix}
        &=& 
        \begin{pmatrix} -A\tran & C\tran \\ -B\tran & D\tran \end{pmatrix}
        \begin{pmatrix} \Sigma_i & 0 \\ 0 & \Sigma_e \end{pmatrix},
        \end{IEEEeqnarray*}
where $\Sigma_i$ is a signature matrix.
\end{theorem}

We now define relaxation systems, the main subject of this paper.
\begin{definition}\label{def:relaxation}
        A system of the form \eqref{eq:system} is said to be a \emph{relaxation system} if
        $D = D\tran \succeq 0$ and $g(t) = Ce^{A t}B$ is a completely
        monotonic function for $t \in [0, \infty)$:
        \begin{IEEEeqnarray*}{+rCl+x*}
        g(t) &=& g(t)\tran \text{ for all } t \geq 0,\\
        (-1)^k \td{^k}{t^k} g(t) &\succeq&  0 \text{ for all } k = 1, 2, \ldots \text{
        and } t \geq 0. & \qedhere
        \end{IEEEeqnarray*}
\end{definition}

Relaxation systems first arose in the context of viscoelastic materials
\autocite{Meixner1964}, and, in the context of electrical circuits, correspond to the
impedances of RC circuits and the admittances of RL circuits.
Several equivalent characterizations of relaxation systems are
known in the literature \autocite{Meixner1964, Willems1972a, Willems1976,
Marcus1975, Bernstein1928, Aissen1951}, which we collect in the following theorem.
\begin{theorem}\label{thm:relaxation_state_space}
        Consider a system of the form \eqref{eq:system}. Then the following are
        equivalent: 
        \begin{enumerate}
                \item the system is a relaxation system.
                \item  $H(s)$ admits the form
        \begin{IEEEeqnarray*}{rCl}
        H(s) &=& G_0 + \frac{G_1}{s} + \sum_{i=2}^n \frac{G_i}{s + \lambda_i},
        \end{IEEEeqnarray*}
        where $G_i = G_i\tran \succeq 0$ for all $i$ and $0 \leq \lambda_0 < \lambda_1 < \ldots
        < \lambda_N$, for some $N \in \mathbb{Z}_{\geq 0}$.
        \item 
        $H(s)$ admits a minimal state space realization $(A_1, B_1, C_1, D_1)$ such that
        \begin{IEEEeqnarray*}{rCl}
        A_1 &=& A_1\tran \preceq 0\\
        B_1 &=& C_1\tran \\
        D_1 &=& D_1\tran \succeq 0.
        \end{IEEEeqnarray*}
        \item $D  \succeq 0$,
        \begin{IEEEeqnarray*}{lCl}
                \begin{pmatrix} CB & CAB & \ldots & CA^{n-1}B \\
                              CAB & CA^2 B & \ldots & CA^n B \\
                              \vdots & \vdots & \ddots & \vdots \\
                              CA^{n-1}B & CA^nB & \hdots & CA^{2n-2}B 
                \end{pmatrix} &\succeq& 0\\
                \begin{pmatrix} CAB & CA^2 B & \ldots & CA^n B \\
                               CA^2B & CA^3 B & \ldots & CA^{n+1} B\\
                              \vdots & \vdots & \ddots & \vdots \\
                              CA^{n}B & CA^{n+1}B & \hdots & CA^{2n-1}B 
                \end{pmatrix} &\preceq& 0,
        \end{IEEEeqnarray*}
        and all three of these matrices are symmetric.
        \end{enumerate}
\end{theorem}

\subsection{Cyclic monotonicity}

In this section, we introduce the notions of monotonicity and cyclic monotonicity,
for operators on a Hilbert space $\hil$.

\begin{definition}
        Given an operator $A: \hil \to \hil$, the \emph{graph of $A$} is the set
        $\gra{A} \subseteq \hil \times \hil$ defined by
        \begin{IEEEeqnarray*}{+rCl+x*}
                \gra{A} &:=& \{ (u, y) \; | \; u \in \hil, y = A(u) \}. & \qedhere
        \end{IEEEeqnarray*}
\end{definition}
\begin{definition}
        An operator $A: \hil \to \hil$ is said to be
        \emph{monotone} if, for all $u_1, u_2 \in \hil, y_1 = A(u_1), y_2 = A(u_2)$,
        \begin{IEEEeqnarray}{rCl}
        \ip{u_1 - u_2}{y_1 - y_2} \geq 0.\label{eq:monotone}
        \end{IEEEeqnarray}
        If $\gra{A}$ is not properly contained within the graph of any other
        monotone operator, $A$ is said to be \emph{maximal monotone}.
\end{definition}

\begin{definition}
        An operator $A: \hil \to \hil$ is said to be
        \emph{$n$-cyclic monotone} if, for all sets of input/output pairs $\{(u_i,
        y_i) \; | \; u_i \in \hil, y_i = A(u_i), i = 0, \ldots, n\}$, 
        \begin{IEEEeqnarray*}{rCl}
        \ip{y_0}{u_0 - u_1} + \ip{y_1}{u_1 - u_2} + \ldots + \ip{y_n}{u_n - u_0} \geq 0.
        \end{IEEEeqnarray*}
        If $A$ is $n$-cyclic monotone  for all $n \geq 1$, $A$ is said to be
        \emph{cyclic monotone}.  If $\gra{A}$ is not contained within the
        graph of any other monotone operator, $A$ is said to be \emph{maximal
        cyclic monotone}.
\end{definition}

Maximality is guaranteed for continuous operators \autocite[Cor.
20.25]{Bauschke2011}, so the Hankel operators associated with  the  stable linear operators considered in 
this paper are automatically maximal.

\begin{definition}
        An operator $A: \hil \to \hil$ is said to be self-adjoint if, for all $u, y
        \in \hil$,
        \begin{IEEEeqnarray*}{+rCl+x*}
                \ip{A (u)}{y} &=& \ip{u}{A (y)}. & \qedhere
        \end{IEEEeqnarray*}
\end{definition}

\textcite{Asplund1970} gives the following characterization of the cyclic
monotonicity of a linear operator.  Given a linear operator $A: \hil \to \hil$, we define
the \emph{complexification of $A$}, denoted $A_c$, by
\begin{IEEEeqnarray*}{rCl}
        A_c(u + jw) := A(u) + jA(w).
\end{IEEEeqnarray*}
This operates on the complexification of $\hil$, denoted $\hil_c$. We endow this
space with the inner product
\begin{IEEEeqnarray*}{rCl}
        \ip{u + jw}{y + jv}_c := \ip{u}{y} + \ip{w}{v} + j(\ip{w}{y} - \ip{u}{v}).
\end{IEEEeqnarray*}
The \emph{numerical range} of an operator $A_c$ on $\hil_c$ is defined as
\begin{IEEEeqnarray*}{rCl}
W(A_c) := \left\{\frac{\ip{A_c (z)}{z}_c}{\norm{z}}\; \middle|\; z \in \dom{A_c}, \norm{z} \neq
0\right\}.
\end{IEEEeqnarray*}

\begin{theorem}[\protect{\textcite[Thm. 3]{Asplund1970}}]
         A linear operator $A$ on $\hil$ is $n$-cyclic monotone if and only if, for
         all $z \in W(A_c)$, $\arg{z} \leq \pi/n$.
\end{theorem}

For the limiting case of cyclic monotonicity, we have the following corollary.

\begin{corollary}\label{cor:cyclic_adjoint}
        A linear operator $A$ on $\hil$ is cyclic monotone if and only if it is
        self-adjoint and, for all $u \in \dom{A}$, $\ip{A(u)}{u} \geq 0$.
\end{corollary}

\begin{proof}
       $n$-cyclic monotonicity for all $n$ implies that $\arg{z} = 0$ for all $z \in
       W(A_c)$.  Equivalently, $\arg{\ip{A_c (z)}{z}} = 0$ for all $z = u + jw \in 
       \dom{A_c}, \norm{z} \neq 0$.  Expanding the inner product:
       \begin{IEEEeqnarray*}{+rCl+x*}
               \arg(\ip{u}{A(u)} + \ip{w}{A(w)} +&&\\
               j(\ip{w}{A(u)} - \ip{u}{A(w)})) &=& 0\\
               \text{so } \ip{u}{A(u)} + \ip{w}{A(w)} &\geq& 0\\
       \text{and } \ip{A(w)}{u} = \ip{w}{A(u)}.&& & \qedhere
       \end{IEEEeqnarray*}
\end{proof}

\begin{definition}
        A function $f: \hil \to \R\cup\{\infty\}$ is said to be \emph{proper} if its value is
        never $-\infty$ and is finite somewhere, \emph{closed} if
        its epigraph is closed and \emph{convex} if, for all $x, y \in \hil$
        and $\theta \in (0, 1)$, 
        \begin{IEEEeqnarray*}{+rCl+x*}
                f(\theta x + (1-\theta)y) &\leq& \theta f(x) + (1-\theta)f(y).&\qedhere
        \end{IEEEeqnarray*}
\end{definition}

Our interest in cyclic monotonicity stems from the following theorem of Rockafellar.
\begin{theorem}[Rockafellar's theorem \protect{\autocite{Rockafellar1966, Rockafellar1970b}}]
        A continuous operator $A: \hil \to \hil$ is maximal cyclic monotone if and only if it is the
        gradient of a closed, convex and proper function from $\hil$
        to $(-\infty, \infty]$.  Moreover, this function is uniquely determined by
        $A$ up to an additive constant.
\end{theorem}

\section{Relaxation and cyclic monotonicity}\label{sec:relaxation_cyclic}

In this section, we establish the relationship between relaxation systems and cyclic
monotone operators, and add a fifth equivalence to
Theorem~\ref{thm:relaxation_state_space}: relaxation is equivalent to cyclic
monotonicity of the Hankel operator.  The following theorem generalizes
\autocite[Cor. 1.2]{Yafaev2015} to multiple input, multiple output operators,
assuming a finite-dimensional state space realization.

\begin{theorem}\label{thm:relaxation_cyclic}
        Consider the system \eqref{eq:system} and assume that $A$ is Hurwitz.  The
        system is a relaxation system if and only if its Hankel operator $\Gamma_h$ is cyclic monotone and $D = D\tran \succeq
        0$.
\end{theorem}

\begin{proof}
        We begin by showing that relaxation implies cyclic monotonicity of the Hankel
        operator (the condition on $D$ being immediate from the definition of
        relaxation).
        By Corollary~\ref{cor:cyclic_adjoint}, cyclic monotonicity of $\Gamma_h$ is
        equivalent to the following two conditions, for all $u, w \in L_2^m$:
        \begin{IEEEeqnarray}{rCl}
        \ip{\Gamma_h w}{u} &=& \ip{w}{\Gamma_h u}\label{eq:adjoint}\\
        \ip{u}{\Gamma_h u}  &\geq& 0.\label{eq:positive}
        \end{IEEEeqnarray}
        We begin by showing \eqref{eq:adjoint}.  Note that relaxation implies
        reciprocity with respect to $\Sigma_e = I$, and this in turn implies symmetry
        of the impulse responses $h(t)$ and $g(t)$.

        We also note that, for any $u, w \in L_2^m$, 
        \begin{IEEEeqnarray*}{rCl}
                && \int_0^\infty u(t)\tran \left(\int_0^\infty h(t + \tau) w(\tau)
                \dd{\tau}\right)\dd{t}\\
                &=& \int_0^\infty u(t)\tran \left(\int_0^\infty g(t + \tau) w(\tau)
                \dd{\tau}\right)\dd{t}. \IEEEyesnumber \label{eq:impulse} 
        \end{IEEEeqnarray*}
Indeed,
        \begin{IEEEeqnarray*}{rCl}
                && \int_0^\infty u(t)\tran \left(\int_0^\infty h(t + \tau) w(\tau)
                \dd{\tau}\right)\dd{t}\\
                &=& \int_0^\infty u(t)\tran \int_0^\infty C e^{At} e^{A\tau} B w(\tau) + D w(\tau) \delta(t + \tau)
                \dd{\tau}\dd{t}\\
                &=&  \int_0^\infty u(t)\tran \left(\int_0^\infty C e^{At} e^{A\tau} B w(\tau) \dd{\tau}\right) \dd{t}  \\
                && + \int_0^\infty u(t)\tran D \bar w(t) \dd{t}, \IEEEyesnumber
                \label{eq:sum}
        \end{IEEEeqnarray*}
        where
        \begin{IEEEeqnarray*}{rCl}
                \bar{w}(t) &:=& \begin{cases} w(t) & t = 0\\
                                                0 & \text{otherwise.}
                                \end{cases}
        \end{IEEEeqnarray*}
        We then have
                \begin{IEEEeqnarray*}{rCl}
                        \int_0^\infty u(t)\tran D \bar w(t) \dd{t} &=& 0,
                \end{IEEEeqnarray*}
        so \eqref{eq:sum} implies \eqref{eq:impulse}.
        Using symmetry of the inner product, \eqref{eq:adjoint} is equivalent to 
        \begin{IEEEeqnarray*}{rCl}
                \int_0^\infty u(t)\tran \left(\int_0^\infty g(t + \tau) w(\tau)
                \dd{\tau}\right)\dd{t}\\
                = \int_0^\infty w(t)\tran \left(\int_0^\infty g(t + \tau) u(\tau)
                \dd{\tau}\right)\dd{t}. \IEEEyesnumber \label{eq:reciprocal}
        \end{IEEEeqnarray*}

        To show that $g(t) = g(t)\tran$ implies \eqref{eq:reciprocal}, take the left
        hand side of \eqref{eq:reciprocal}, transpose and apply Fubini's theorem:
        \begin{IEEEeqnarray*}{rCl}
                \int_0^\infty u(t)\tran \left(\int_0^\infty g(t + \tau) w(\tau)
                \dd{\tau}\right)\dd{t}\\
                = \int_0^\infty \left(\int_0^\infty w(\tau)\tran g(t + \tau)\tran
                \dd{\tau}\right)\, u(t)
                \dd{t}\\
                = \int_0^\infty \left(\int_0^\infty w(\tau)\tran g(t + \tau)\tran u(t) \dd{t}\right)
                \dd{\tau}\\
                = \int_0^\infty w(t)\tran \left(\int_0^\infty g(t + \tau)u(t) \dd{t}\right)
                \dd{\tau}.
        \end{IEEEeqnarray*}

        We next show that relaxation implies \eqref{eq:positive}. Let $(A_1, B_1, C_1, D_1)$ be a state space
        realization of the form of Theorem~\ref{thm:relaxation_state_space}, 3), with
        impulse response $h(t)$.  Then, using \eqref{eq:impulse},
        \begin{IEEEeqnarray*}{rCl}
        \ip{u}{\Gamma_h u} &=& \int_0^\infty u\tran(t) \int_0^\infty h(t + \tau)
        u(\tau) \dd{\tau} \dd{t}\\\IEEEyesnumber\label{eq:Henk}
                           &=& \int_0^\infty u\tran(t) \int_0^\infty C_1e^{A_1t} e^{A_1
                           \tau} B_1 u(\tau)
                           \dd{\tau}\dd{t}\\
                           &=& \left(\int_0^\infty e^{A_1t} B_1 u(t) \dd{t} \right)\tran \int_0^\infty e^{A_1t} B_1
        u(t) \dd{t}\\
                           &\geq& 0.
        \end{IEEEeqnarray*}
        This establishes that relaxation implies cyclic monotonicity of the Hankel operator.

        We now show the converse, that cyclic monotonicity of the Hankel operator and
        $D = D\tran \succeq 0$ together imply relaxation.
        We begin by showing that \eqref{eq:reciprocal} implies symmetry of
        $g(t)$ for all $t \geq 0$.  Indeed, let 
        $v(\tau) = \delta(\tau)e_j$
        and $u(t) = \delta(t - t_0) e_i$, where $t_0 \in [0, \infty)$, $e_i$ denotes
        the $i^{\text{th}}$ canonical basis vector of $\R^n$ and $\delta$ denotes the Dirac
        delta.  Substituting these signals into \eqref{eq:reciprocal} gives
        \begin{IEEEeqnarray*}{rCl}
        e_i\tran g(t_0) e_j &=& e_j\tran g(t_0) e_i,
        \end{IEEEeqnarray*}
        that is, $g(t_0)$ is symmetric for all $t_0 \in [0, \infty)$, which is
        equivalent to symmetry of $h(t)$ under the assumption $D = D\tran$.  This in
        turn is equivalent to reciprocity with respect to $\Sigma_e = I$.

        Finally, we show that reciprocity, \eqref{eq:positive} and $D = D\tran \succeq 0$ imply relaxation.
        Let  be a stable system with $\hat D = \hat
        D \tran \succeq 0$ and Hankel operator
        $\Gamma_h$ which satisfies \eqref{eq:positive} and \eqref{eq:reciprocal}.
        Let $(\hat A, \hat B, \hat C, D)$ be a minimal system with transfer function
        equal to $D\delta(t) + Ce^{At}B$. 
        By reciprocity, it follows from \autocite[Lem 3]{Willems1972a} that there exists a 
        unique, invertible, symmetric matrix $T$ such that
        \begin{IEEEeqnarray*}{rCl}
        \hat A\tran T &=& T \hat A\\
        T\hat B &=& \hat C\tran.
        \end{IEEEeqnarray*}
        We claim that $T \geq 0$.  Suppose, on the contrary, that $T$ has a negative
        eigenvalue.  Let $x_0$ be a corresponding eigenvector.  Let $\bar{u}:
        (-\infty, 0] \to \R^n$ be an input that drives the system from $x=0$ at
        $t=-\infty$ to $x(0) = x_0$.  Such an input exists, as $(\hat A, \hat B)$ is
        controllable.  Let $u(t) = \bar{u}(-t)$.  By positivity of $\Gamma_h$, we
        have
        \begin{IEEEeqnarray*}{rCl}
        0 &\leq& \ip{u}{\Gamma_h u}\\
          &=& \int_0^\infty u(t)\tran \int_0^\infty \hat Ce^{\hat A(t +
          \tau)}\hat Bu(\tau)\dd{\tau}\dd{t}\\
          &=& \int_0^\infty u(t)\tran \hat Ce^{\hat At} \int_{-\infty}^0 e^{-\hat A\tau} \hat B
          \bar{u}(\tau) \dd{\tau}\dd{t}\\
          &=& \int_0^\infty u(t)\tran \hat C e^{\hat At} x_0 \dd{t}\\
          &=& \int_0^\infty u(t)\tran \hat B\tran T e^{\hat At} x_0 \dd{t}\\
          &=& \int_{-\infty}^0 \bar{u}(t)\tran \hat B\tran e^{-\hat A\tran t} \dd{t} T x_0\\
          &=& x_0\tran T x_0 < 0,
        \end{IEEEeqnarray*}
        which is a contradiction.  Hence $T \geq 0$.  It follows from
        Lemma~\ref{lem:passivity} in the Appendix that the system is passive.  It then follows from
        \autocite[Thm. 7]{Willems1972a} that there exists a minimal realization $(A_1, B_1,
        C_1, D_1)$ of the system which satisfies
        \begin{IEEEeqnarray*}{rCl}
        \Sigma_i A_1 &=& A\tran_1 \Sigma_i\\
        C_1\tran &=& -\Sigma_i B_1\\
        D_1 &=& D_1\tran \succeq 0,
        \end{IEEEeqnarray*}
        where $\Sigma_i$ is a signature matrix.  
        It follows from Equation~\eqref{eq:Henk} and positivity of $\Gamma_h$
        that 
        \begin{IEEEeqnarray}{rCl}
        \int_0^\infty u(t)\tran C_1 e^{At} \dd{t} \int_0^\infty e^{A_1\tau} B_1 u(\tau)
        \dd{\tau} \geq 0\label{eq:contra}
        \end{IEEEeqnarray}
        for all $u$.  Hence
        \begin{IEEEeqnarray*}{rCl}
        -\int_0^\infty u(t)\tran B_1\tran e^{A_1\tran t} \dd{t} \Sigma_i \int_0^\infty
        e^{A_1\tau} B_1 u(\tau) \dd{\tau} \geq 0
        \end{IEEEeqnarray*}
        for all $u$.
        Suppose that $\Sigma_i$ has  entry $(j, j)$ equal to $1$.  By controllability
        of $(A_1, B_1)$, we can choose an input such that
        \begin{IEEEeqnarray*}{rCl}
        \int_0^\infty e^{A_1\tau} B_1 u(\tau) \dd{\tau} = e_j.
        \end{IEEEeqnarray*}
        But then $-e_j\tran \Sigma_i e_j \prec 0$, which contradicts
        \eqref{eq:contra}.  Hence $\Sigma_i = -I$, so the system is of the relaxation
        type. 
\end{proof}

\section{Intrinsic storages for relaxation systems}\label{sec:LTI_storage}

Theorem~\ref{thm:relaxation_cyclic} establishes the equivalence of relaxation and
cyclic monotonicity of the Hankel operator.  It then follows from Rockafellar's theorem
that the Hankel operator is the gradient of a closed, convex and proper functional
mapping $L_2^m \to \R$. It
turns out that this convex functional is precisely the input/output storage
observed by \textcite[$\S10$]{Willems1972a}.

Before formalizing this result, we show that passivity is guaranteed by the existence
of a nonnegative functional of the past input to the system.  We call this object an
\emph{intrinsic storage functional}.
We then give a simple, illustrative example.

\begin{proposition}\label{prop:storage}
         Consider a system of the form \eqref{eq:system}.  Given a signal $u \in L_2(\R,
         \R^m)$ and time $t\in\R$, denote by $u_t$ the truncation of $u$ to the time axis $(-\infty,
         t]$.  If there exists a functional $V$ mapping a truncated signal $u_t$ into
         $\R_{\geq 0}$ and satisfying
         \begin{IEEEeqnarray}{rCl}
                 \td{V}{t}(u_t) \leq u(t)\tran y(t),\label{eq:passive}
         \end{IEEEeqnarray}
        for all $t \in \R$ and input/output trajectories $(u, y)$ of the system, then
        the system is passive.
\end{proposition}

\begin{proof}
        Let $t_0, t_1 \in \R$, $t_1 \geq t_0$.  Integrating \eqref{eq:passive} from
        $t_0$ to $t_1$ gives
        \begin{IEEEeqnarray*}{rCl}
                V(u_{t_0}) - V(u_{t_1}) \geq -\int_{t_0}^{t_1} u(t)\tran y(t) \dd{t}.
        \end{IEEEeqnarray*}
        Passivity then follows from nonnegativity of $V(u_{t_1})$, with $K$ in
        Definition~\ref{def:passive} equal to $V(u_{t_0})$.
\end{proof}

\begin{example}\label{ex:rc}
        Consider the linear RC circuit shown in Figure~\ref{fig:rc}.  Denoting the
        voltage on the capacitor by $v_c$, we have the following state space model
        for the impedance of the circuit: 
        \begin{IEEEeqnarray*}{rCl}
        \td{}{t} v_c &=& \frac{-1}{R_1 C} v_c + \begin{pmatrix} \frac{1}{C} & \frac{1}{C} \end{pmatrix} \begin{pmatrix} i_1 \\ i_2 \end{pmatrix},\\
        \begin{pmatrix} v_1 \\ v_2 \end{pmatrix} &=& \begin{pmatrix} 1 \\ 1\end{pmatrix} v_2 + \begin{pmatrix} 0 & 0 \\ 0 & R_2 \end{pmatrix} 
                \begin{pmatrix} i_1 \\ i_2 \end{pmatrix}.
        \end{IEEEeqnarray*}
        \begin{figure}[hb]
                \centering
                \includegraphics{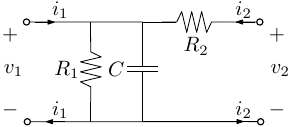}
        \caption{A two-port RC circuit.}
                \label{fig:rc}
        \end{figure}
        We consider the following experiment: time-varying current sources,
        $\bar{i}_{1t}(\cdot)$ and $\bar{i}_{2t}(\cdot)$, are
        attached to the ports from time $-\infty$, when there is no charge on the capacitor, to
        time $t \in \R$.  The current sources are then replaced by voltmeters, which read 
        voltages $\bar{v}_1(\cdot)$ and $\bar{v}_2(\cdot)$.  We define $i_{nt}(\tau) = \bar{i}_{nt}(t - \tau)$ and
        $v_n(\zeta) = \bar{v}_n(\zeta + t)$ for $n = 1,2$. Define $v =
        \begin{pmatrix} v_1 & v_2 \end{pmatrix}\tran$ and $i_t =
        \begin{pmatrix} i_{1t} & i_{2t} \end{pmatrix}\tran$.  Solving the state space
        model gives the Hankel operator
        \begin{IEEEeqnarray*}{rCl}
        v(\zeta) &=&  \int_0^\infty \begin{pmatrix}1 \\ 1\end{pmatrix} e^{\frac{-1}{R_1C}(\zeta + \tau)} \begin{pmatrix} \frac{1}{C} & \frac{1}{C}\end{pmatrix} i_t(\tau) \dd{\tau},
        \end{IEEEeqnarray*}
plus an additional term $R_2 i_{2t}(0)$ when $\zeta = 0$.
        Computing the inner product $(1/2)\ip{i_t}{v}$ over $L_2$ gives
        \begin{IEEEeqnarray*}{lCl}
        \frac{1}{2}\int_0^\infty i_t(\zeta)\tran \int_0^\infty \begin{pmatrix} 1\\ 1\end{pmatrix}
        e^{\frac{-1}{R_1C}(\zeta + \tau)} \begin{pmatrix} \frac{1}{C} & \frac{1}{C}\end{pmatrix} i_t(\tau)\dd{\tau}\dd{\zeta}\\
        = 
        \frac{1}{2C}\left( \int_0^\infty (i_{1t}(\zeta) + i_{2t}(\zeta))e^{\frac{-1}{R_1C}\zeta}\dd{\zeta}\right.\\
         \quad\qquad\left.\int_0^\infty (i_{1t}(\tau) + i_{2t}(\tau))
        e^{\frac{-1}{R_1C}\tau}\dd{\tau}\right)\\
        = \frac{1}{2C} q_c(0)^2,
        \end{IEEEeqnarray*}
        where $q_c = \frac{1}{C} v_c$ is the charge on the capacitor and the last
        line follows by solving the state space equations with zero initial
        condition.
        This expression is the energy stored in the capacitor at time $\tau = 0$.
        Taking the derivative with respect to time gives
        \begin{IEEEeqnarray*}{rCl}
        \td{}{t}\frac{1}{2}\ip{i_t}{v} &=& \frac{1}{C} q_c(0)\td{}{t} q_c(0).
        \end{IEEEeqnarray*}
        Let $\eta(t, \tau) := t - \tau$.  Then
        \begin{IEEEeqnarray*}{*l+rCl}
                &\td{}{t} i_{nt} &=& \td{}{t} \bar{i}_{nt}(\eta(t, \tau)) = \td{\bar{i}_{nt}}{\eta} \td{\eta}{t} =  \bar{i}_{nt}^\prime (t - \tau)\\
                \text{and}& \td{}{\tau} i_{nt} &=& \td{}{\tau} \bar{i}_{nt}(\eta(t, \tau)) = \td{\bar{i}_{nt}}{\eta} \td{\eta}{\tau} = - \bar{i}_{nt}^\prime (t - \tau),
        \end{IEEEeqnarray*}
        so $\td{}{t} i_{nt} =  -\td{}{\tau}
        i_{nt}$. We then have:
        \begin{IEEEeqnarray*}{rCl}
        \td{}{t} q_c(0) &=& \int_0^\infty e^{\frac{-1}{R_1C}\tau} \td{}{t}(i_{1t}(\tau) + i_{2t}(\tau))
        \dd{\tau} \\
                       &=& -\int_0^\infty e^{\frac{-1}{R_1C} \tau} \td{}{\tau} (i_{1t}(\tau) + i_{2t}(\tau))
                               \dd{\tau}.
        \end{IEEEeqnarray*}
        Integrating by parts then gives
        \begin{IEEEeqnarray*}{rCl}
        \td{}{t}q_c(0) &=& -\left[e^{\frac{-1}{R_1C}\tau}(i_{1t}(\tau) + i_{2t}(\tau))\right]_0^\infty \\&&-
        \frac{1}{R_1C}\int_0^\infty e^{\frac{-1}{R_1C}\tau} (i_{1t} + i_{2t})(\tau)\dd{\tau}\\
                   &=& i_{1t}(0) + i_{2t}(0) - \frac{1}{R_1C} q_c(0), \text{ so }\\
        \td{}{t}\frac{1}{2}\ip{i_t}{v} &=& v_c(0)(i_{1t}(0) + i_{2t}(0)) -
        \frac{1}{R_1C^2}q_c(0)^2\\
                                       &\leq& v_c(0) (i_{1t}(0) + i_{2t}(0)) + R_2 i_{2t}(0)^2\\
                                       &=& \bar v(t)\bar{i}_t(t).
        \end{IEEEeqnarray*}
        The variables $\bar v$ and $\bar i_t$ correspond to a particular experiment, however, the
        right hand side of this 
        dissipation inequality only involves the value of $\bar i_t$ and $\bar v$ at
        time $t$, the instant in the experiment when both the current source and the
        voltmeter are connected.  These can thus be considered samples of an
        arbitrary current/voltage trajectory.
        The functional $(1/2)\ip{i_t}{v}$ is thus an intrinsic storage functional for the system,
        and is expressed purely in terms of the input $i$ and output $v$.
        Furthermore, the derivative of this functional with respect to $i_t$ is
        the Hankel operator of the system.  The quantity $(1/R_1C^2)q_c(0)^2$ is the
        instantaneous power dissipated by the resistor $R_1$.
\end{example}

In order to generalize the construction of the intrinsic storage in Example~\ref{ex:rc} to
arbitrary relaxation systems, we require a notion of gradient on $L_2^m$.  This is
given by the functional derivative, $\partial V/\partial u$, which we define via the
first variation:
        \begin{IEEEeqnarray*}{rCl}
                \ip{\pd{V}{u}}{\phi} &:=& \left[\td{}{\epsilon}\left(V(u + \epsilon \phi)\right)\right]_{\epsilon = 0}.
\end{IEEEeqnarray*}

\begin{lemma}\label{lem:gradient}
        Let $h$ be the impulse response of a relaxation system, and 
        $\Gamma_h$ be the corresponding Hankel operator.  Then
        $\Gamma_h$ is the functional derivative of 
        \begin{IEEEeqnarray*}{rCl}
        V(u) := \frac{1}{2}\ip{u}{\Gamma_h u}.
        \end{IEEEeqnarray*}
\end{lemma}

\begin{proof}
Computing the functional derivative gives
        \begin{IEEEeqnarray*}{rCl}
                                    \ip{\pd{V}{u}}{\phi} &=& \frac{1}{2}\left[\td{}{\epsilon} \ip{u + \epsilon \phi}{\Gamma_h(u + \epsilon \phi)}\right]_{\epsilon = 0}\\
                                     &=& \frac{1}{2}\ip{\phi}{\Gamma_h u} + \frac{1}{2}\ip{u}{\Gamma_h \phi}\\
                                     &=& \ip{\Gamma_h u}{\phi},
        \end{IEEEeqnarray*}
        where the final inequality follows from self-adjointness of $\Gamma_h$.  It
        then follows that $\partial V/ \partial u = \Gamma_h$.
\end{proof}

The following theorem establishes that the function of Lemma~\ref{lem:gradient} is in
fact an intrinsic storage functional.

\begin{theorem}\label{thm:linear_dissipation}
        Let $h$ be the impulse response of a relaxation system, and 
        $\Gamma_h$ be the corresponding Hankel operator.  Then 
        the system is passive with intrinsic storage functional
        \begin{IEEEeqnarray*}{rCl}
        V(u) := \frac{1}{2}\ip{u}{\Gamma_h u}.
        \end{IEEEeqnarray*}
\end{theorem}

The proof of Theorem~\ref{thm:linear_dissipation} makes use of the following lemma,
which establishes a recursive property of relaxation systems with respect to the
derivative.  This is a generalization of the fact that the power dissipated by the
resistor $R_1$ in
Example~\ref{ex:rc} is positive.

\begin{lemma}\label{lem:gradient_cyclic}
        Let $g(t) = Ce^{At}B$ be the impulse response of a relaxation
        system, without the direct component $D\delta(t)$.  Then any
        system with impulse response $-\td{}{t} g$ is also a relaxation system.
\end{lemma}

\begin{proof}
        By Definition~\ref{def:relaxation}, $g$ is completely monotonic, so 
        \begin{IEEEeqnarray*}{rCl}
                (-1)^k \td{^k}{t^k} g(t) \geq 0
        \end{IEEEeqnarray*}
        for all $k = 1, 2, \ldots$ This implies complete monotonicity of $-\td{}{t}
        g$.
\end{proof}
\begin{proof}[Proof of Theorem~\ref{thm:linear_dissipation}]
        Nonnegativity of $V$ follows from positivity of $\Gamma_h$
        (Theorem~\ref{thm:relaxation_cyclic}).  It remains to show that $V$ satisfies
        the dissipation inequality \eqref{eq:passive}.
        Let the input trajectory be $\bar{u} \in L_{2}(\R, \R^m)$
        and define the past input corresponding to time $t \in \R$ by
        \begin{IEEEeqnarray*}{rCl}
        u_t(\tau) &:=& \bar{u} (t - \tau), \quad \tau \in [0, \infty).
        \end{IEEEeqnarray*}
        Let $\eta(t, \tau) := t - \tau$.  Then
        \begin{IEEEeqnarray*}{*l+rCl}
                &\td{}{t} u_t(\tau) &=& \td{}{t} \bar{u}_t(\eta(t, \tau)) = \td{\bar{u}_t}{\eta} \td{\eta}{t} =  \bar{u}_t^\prime (t - \tau)\\
                \text{and}& \td{}{\tau} u_t &=& \td{}{\tau} \bar{u}_t(\eta(t, \tau)) = \td{\bar{u}_t}{\eta} \td{\eta}{\tau} =  - \bar{u}_t^\prime (t - \tau),\\
        \text{so} & 
        \td{}{t} u_t &=&  -\td{}{\tau}
        u_t.\IEEEyesnumber\label{eq:u_t}
\end{IEEEeqnarray*}
        We then have
        \begin{IEEEeqnarray*}{rCl}
                \td{V}{t}(u_t) &=& \ip{\pd{V}{u}(u_t)}{\pd{u_t}{t}}\\
                          &=& \int_0^\infty \int_0^\infty u_t(\tau)\tran h(\zeta + \tau) \dd{\tau} \pd{u_t}{t}(\zeta) \dd{\zeta}\\
                          &=& -\int_0^\infty y(\zeta)\tran \pd{u_t}{\zeta} (\zeta) \dd{\zeta}, 
        \end{IEEEeqnarray*} 
        where the final line uses Lemma~\ref{lem:gradient} and
        Equation~\eqref{eq:u_t}.  Integration by parts then gives
        \begin{IEEEeqnarray*}{lCl}
        \td{V}{t}(u_t) = -\left[y(\zeta)\tran u_t(\zeta)\right]_0^\infty + \int_0^\infty \td{}{\zeta}y(\zeta)\tran u_t(\zeta) \dd{\zeta} \\
        = y(0)\tran u_t(0) + \\ \int_0^\infty\hspace{-6pt} \int_0^\infty u_t(\tau)\tran \td{h}{\zeta}(\tau +
        \zeta)\dd{\tau} u_t(\zeta) \dd{\zeta}, \IEEEyesnumber \label{eq:parts}
        \end{IEEEeqnarray*}
        where \eqref{eq:parts} uses \eqref{eq:impulse} in the proof of Thm.
        \ref{thm:relaxation_cyclic}.
        Denote $\mathrm{d}g/\mathrm{d}\zeta$ by $g^\prime$.  Then the rightmost term in
        \eqref{eq:parts}
        can be written as
        \begin{IEEEeqnarray}{rCl}
        -\ip{\Gamma_{(-g^\prime)} u_t}{u_t} \leq 0,\label{eq:gradient_ip}
        \end{IEEEeqnarray}
        where the inequality follows from the fact that that $\Gamma_{(-g^\prime)}$ is
        the Hankel operator of a relaxation system (Lemma~\ref{lem:gradient_cyclic}),
        hence cyclic monotone
        (Theorem~\ref{thm:relaxation_cyclic}). Substituting in
        \eqref{eq:parts} gives
        \begin{IEEEeqnarray*}{+rCl+x*}
        \td{V}{t}(u_t) &\leq& u_t(0)y(0) = \bar u_t(t)\tran \bar y(t).&\qedhere
        \end{IEEEeqnarray*}
\end{proof}

A consequence of Rocakfellar's theorem is that the storage $V(u_t)$ is uniquely
determined by the Hankel operator $\Gamma_h$, up to an additive constant.  It was
observed in \autocite{Willems1972a} that this same storage is also uniquely determined
by the requirements of passivity and internal reciprocity.

\section{Conclusions}\label{sec:conclusions}

We have shown that a system being of the relaxation type is equivalent to cyclic
monotonicity of the Hankel operator.  Rockafellar's theorem allows us to construct a
convex storage functional, whose gradient is the Hankel operator, which is completely
determined by input/output measurements.

Cyclic monotonicity is equally well-defined for the Hankel operators of nonlinear
systems, and this allows us to construct intrinsic storages for nonlinear systems.
This will be a topic of future research.

\appendix

\begin{lemma}\label{lem:passivity}
        Consider a stable system  of the form \eqref{eq:system}.
        Suppose that $D = D\tran  \succeq 0$ and there exists a matrix $T = T\tran \succeq 0$ such that
        \begin{IEEEeqnarray*}{rCl}
        A\tran T &=& TA \\
        T B &=& C\tran.
        \end{IEEEeqnarray*}
        Then the system is passive.
\end{lemma}

\begin{proof}
       It suffices to show that $T$ satisfies \eqref{eq:lmi}, which reduces to $TA
       \preceq 0$ since $TA$ is symmetric.
        We can factorize $T$ as follows:
        \begin{IEEEeqnarray*}{rCl}
        T = V\tran \begin{pmatrix} \Lambda_1 & 0 \\ 0 & 0 \end{pmatrix}V,
        \end{IEEEeqnarray*}
        where $\Lambda_1 > 0$ and $V$ is orthogonal, so
        \begin{IEEEeqnarray*}{rCl}
        VTAV\tran &=& \begin{pmatrix} \Lambda_1 & 0 \\ 0 & 0 \end{pmatrix}VAV\tran.
        \end{IEEEeqnarray*}
        Since $TA$ is symmetric, $VTAV\tran = (VTAV\tran)\tran$.  Define $\bar A
        := VAV\tran$.  Note that $\bar{A}$ is Hurwitz, as $A$ is Hurwitz.
        Partition $\bar{A}$ into
        \begin{IEEEeqnarray*}{rCl}
        \bar{A} &=& \begin{pmatrix} \bar{A}_{11} & \bar{A}_{12} \\
                                    \bar{A}_{21} & \bar{A}_{22}.
                            \end{pmatrix}
        \end{IEEEeqnarray*}
        Then 
        \begin{IEEEeqnarray*}{rCl}
        \begin{pmatrix} \Lambda_1 & 0 \\ 0 & 0 \end{pmatrix}\bar{A} &=&
        \begin{pmatrix} \Lambda_1 \bar{A}_{11} & \Lambda_2 \bar{A}_{12} \\
                                    0 & 0
                            \end{pmatrix}.
        \end{IEEEeqnarray*}
        Since this matrix is symmetric, $\Lambda_1 \bar{A}_{12} = 0$, which implies
        $\bar{A}_{12} = 0$.  Hence $\bar{A}$ is lower block triangular, and so
        $\bar{A}_{11}$ is Hurwitz.  We claim that $\Lambda_1 \bar{A}_{11} < 0$.
        
        Let $\operatorname{In}(A)$ denote the inertia of the matrix $A$.
        Since $\Lambda_1 \bar{A}_{11}$ is symmetric, it follows from Sylvester's law
        of inertia that
        \begin{IEEEeqnarray*}{rCl}
        \operatorname{In}(\Lambda_1\bar{A}_{11}) &=&
        \operatorname{In}(\Lambda_1^{-\frac{1}{2}} \Lambda_1 \bar{A}_{11}
        \Lambda_1^{-\frac{1}{2}})\\
         &=& \operatorname{In}(\Lambda_1^{\frac{1}{2}} \bar{A}_{11}
        \Lambda_1^{-\frac{1}{2}}).
        \end{IEEEeqnarray*}
        This equals the inertia of $\bar{A}_{11}$, so we have that all the
        eigenvalues of $\bar{A}_{11}$ are real and negative, and $\Lambda_1
        \bar{A}_{11} \prec 0$.  Hence
        \begin{IEEEeqnarray*}{rCl}
                \begin{pmatrix} \Lambda_1 & 0 \\ 0 & 0 \end{pmatrix} \preceq 0,
        \end{IEEEeqnarray*}
        so $TA \preceq 0$.
\end{proof}

\printbibliography
\end{document}